\documentclass[12pt,reqno,a4paper]{amsart}
\usepackage{
    amsmath,  amsfonts, amssymb,  amsthm,   amscd,
    gensymb,  graphicx, comment,  etoolbox, url,
    booktabs, stackrel, mathtools,enumitem, mathdots,  microtype, lmodern,    mathrsfs, graphicx, tikz,  longtable,tabularx, float, tikz, pst-node, tikz-cd, multirow, tabularx, amscd,  bm, array, makecell, diagbox, booktabs,ragged2e, caption, subcaption }
\usepackage{makecell,slashbox}

\usepackage{xcolor}
\usepackage[utf8]{inputenc}
\usepackage{microtype, fullpage, wrapfig,textcomp,mathrsfs,csquotes,fbb}
\usepackage[colorlinks=true, linkcolor=blue, citecolor=blue, urlcolor=blue, breaklinks=true]{hyperref}
\usepackage[capitalise]{cleveref}
\usepackage{todonotes}
\usetikzlibrary{positioning}
\usetikzlibrary{shapes,arrows.meta,calc}

\usetikzlibrary{arrows}

\newtheorem{theorem}{Theorem}[section]

\newtheorem{conjecture}[theorem]{Conjecture}
\newtheorem{corollary}[theorem] {Corollary}
\newtheorem{definition}[theorem]{Definition}
\newtheorem{example}[theorem]{Example}
\newtheorem{lemma}[theorem]{Lemma}

\newtheorem{proposition}[theorem]{Proposition}
\newtheorem{remark}[theorem]{Remark}

\def\S{\mathbb{S}}

\newcommand\R{\mathbb{R}}
\newcommand\Z{\mathbb{Z}}
\newcommand{\cs}{\mathcal{S}}

\newcommand{\co}{\mathrm{coind}}
\newcommand{\ind}{\mathrm{ind}}
\newcommand{\hts}{\mathrm{ht}}
\def\mbar{\overline{M}}
\newcommand{\malpha}{\mathrm{M}_{\alpha}}
\newcommand{\mbalpha}{\overline{\mathrm{M}}_{\alpha}}
\newcommand{\colim}{\mathrm{colim}}
\newcolumntype{x}[1]{>{\centering\arraybackslash}p{#1}}

\begin{document}
\title{The Borsuk-Ulam theorem for planar polygon spaces}
\author{Navnath Daundkar}
\address{Indian Institute of Technology Bombay}
\email{navnathd@iitb.ac.in}
\author{Priyavrat Deshpande}
\address{Chennai Mathematical Institute}
\email{pdeshpande@cmi.ac.in}
\author{Shuchita Goyal}
\address{Indian Institute of Technology Delhi}
\email{shuckriya.goyal@gmail.com}
\author{Anurag Singh}
\address{Indian Institute of Technology Bhilai}
\email{anurags@iitbhilai.ac.in}

\thanks{The second author is partially supported by the MATRICS grant MTR/2017/000239 and  a grant from the Infosys Foundation. 
The third author is supported by DST, India.}

\begin{abstract}
The moduli space of planar polygons with generic side lengths is a closed, smooth manifold. 
Mapping a polygon to its reflected image across the $X$-axis defines a fixed-point-free involution on these moduli spaces, making them into free $\Z_2$-spaces. 
There are some important numerical parameters associated with free $\Z_2$-spaces, like index and coindex.
In this paper, we compute these parameters for some moduli spaces of polygons. 
We also determine for which of these spaces a generalized version of the Borsuk-Ulam theorem hold. 
Moreover, we obtain a formula for the Stiefel-Whitney height in terms of the genetic code, a combinatorial data associated with side lengths.
\end{abstract}
\maketitle

\section{Introduction}\label{sec:intro}

The Borsuk-Ulam (BU) theorem has been an object of central attraction in topology. 
It states that any continuous map from the $d$-sphere ${S}^d$ to the Euclidean space $\mathbb{R}^d$
must identify a pair of antipodal points.
Recently, the BU theorem has been studied for many different complexes with a free $\Z_2$-action. 
For instance, Musin \cite{BUmfds} considered PL-manifolds, Gonçalves et al. \cite{BUlowerdim} considered finite dimensional CW-complexes with a free cellular involution. 

For a topological space $X$ with a fixed-point-free involution $\nu$, we say that $(X,\nu, d)$ is a BU-triple if for every continuous map	$f : X \rightarrow \R^d$ there exists $x \in X$ such that $f(x) = f(\nu (x))$. 
The index of a free $\Z_2$-space $X$ is the smallest number $n$ such that there is a $\Z_2$-equivariant map from $X$ to $S^n$.
It is known that 
if the index of $X$ is $d$ then $(X, \mu,d)$ is always a BU-triple and so is $(X, \mu,t)$ for any $t\leq d$. 
One way of determining BU-triples is to compute a index and other related numerical quantities.
The main focus of the paper is to try and calculate the exact value of the index for a class of spaces. 
In some cases where the exact value is not possible to calculate we provide bounds in terms of other related quantities which we now define. 

In \cite{BUlowerdim}, the authors used the index and the Stiefel-Whitney height (see \Cref{coind}) to find BU-triples. 
There is also the dual notion of index of a free $\Z_2$-space, the coindex.
A $\Z_2$-space is called {\itshape tidy} if all these three numerical parameters coincide; otherwise it is called {\itshape non-tidy}. 
Many authors have considered the problem of finding free $\Z_2$-spaces that are tidy. 
For more discussion about the index, coindex and tidy spaces, the reader is referred to the book of Matou\v{s}ek \cite[Section 5.3]{UBUthm} and Csorba's thesis \cite{csorba}.

In this article, we investigate the tidyness and existence of BU-triples among the class of moduli spaces of planar polygons with a free cellular involution. 
The \emph{moduli space of planar polygons} associated with a length vector $\alpha=(\alpha_{1},\dots, \alpha_{n})$, denoted by $\mathrm{M}_{\alpha}$, is the collection of all closed piecewise linear paths in the plane up to orientation preserving isometries with side lengths $\alpha_{1}, \alpha_{2},\dots, \alpha_{n}$. Equivalently, we can define $\mathrm{M}_{\alpha}$ as
\[\mathrm{M}_{\alpha}= \{(v_{1},v_{2},\dots,v_{n})\in (S^{1})^{n} : \sum_{i=1}^{n}\alpha_{i}v_{i} = 0 \}/\mathrm{SO}_{2},\]
where $S^{1}$ is the unit circle and the group of orientation preserving isometries $\mathrm{SO}_{2}$ acts diagonally. 
The moduli space of planar polygons (associated with $\alpha$) viewed upto isometries is defined as \[\overline{\mathrm{M}}_{\alpha}= \{(v_{1},v_{2},\dots,v_{n})\in (S^{1})^{n} : \displaystyle\sum_{i=1}^{n}\alpha_{i}v_{i} = 0 \}/\mathrm{O}_{2},\]
where the group of all isometries of the plane $\mathrm{O}_{2}$ acts diagonally.
If we choose a length vector $\alpha$ such that $\sum_{i=1}^{n}\pm \alpha_{i} \neq 0$ then the moduli spaces $\mathrm{M}_{\alpha}$ and $\overline{\mathrm{M}}_{\alpha}$ are closed, smooth manifolds of dimension $n-3$. 
Such length vectors are called \emph{generic}.
In this paper the length vectors are assumed to be generic unless stated otherwise.

Moduli spaces of planar polygons have been studied extensively. 
See, for example, Farber's book \cite{zbMATH05315240}, which contains many results that express topological invariants of the moduli space in terms of the combinatorial data associated with the length vector.
Hausmann and Knutson \cite{cohomologyring} computed the mod $2$ cohomology ring of $\mbalpha$. 
Using this presentation of the cohomology ring, Kamiyama \cite{KamiyamaSWheight} computed the height of the Stiefel-Whitney class for (generic) $\alpha=(1,\dots,1,r)$, where $r$ and $n$ have the same parity. 
Panina \cite{GP} showed that the orientation preserving moduli spaces admit a CW-structure with free $\Z_2$-action. 
It is therefore natural to look for BU triples among these spaces and also identify which one of these are tidy. 
The aim of this article is to achieve these two goals. 

The paper is organized as follows. 
In \Cref{sec:prelim} we collect relevant results about $\Z_2$-spaces and polygon spaces. 
\cref{sec3} contains results about those moduli spaces for which the genetic code of the length vector contains exactly one gene. 
Then in \cref{swheight} we consider the two genes case and mainly focus on deriving an expression for the Stiefel-Whitney height. 
The formula we obtain generalizes a result of Don Davis. 
Finally, in \cref{qepolys} we tackle the class of quasi-equilateral polygons. 

\section{Preliminaries}\label{sec:prelim}

Let $X$ be a topological space with a free $\Z_2$-action and consider the $n$-sphere, $S^n_a$, with the antipodal action.	
Then we have the following numerical data associated with $\Z_2$-spaces.

\begin{definition}\label{coind}
\emph{The \emph{coindex} of $X$ is
\[\mathrm{coind}(X) := \mathrm{max} \{n\geq 0 : \exists~  \textit{$\Z_2$-map $S^n_a\to X$}\}. \]
The \emph{index} of $X$ is
\[\mathrm{ind}(X) := \mathrm{min} \{n\geq 0 : \exists~  \textit{$\Z_2$-map $X \to S^n_a$}\}. \]
The Stiefel-Whitney \emph{height} of $X$ is
\[\mathrm{ht}(X) := \mathrm{sup} \{n\geq 0 : (w_1(X))^n\neq 0 \}, \]
where $w_1(X)$ is the first Stiefel-Whitney class of the double cover $X\to X/\mathbb{Z}_2$.}
\end{definition}
For a free $\Z_2$-space $X$, the following inequality relates these three parameters:
\begin{equation}\label{eq: coind ht ind}
0 \le \mathrm{coind}(X) \le \mathrm{ht}(X) \le \mathrm{ind}(X) \le \mathrm{dim}(X).
\end{equation}
Below is an equivalent formulation of BU-triples in terms of above defined numerical parameters. 
The proof of this equivalence can be found in \cite{BUlowerdim}. 

\begin{proposition}[{\cite[Theorem 3.4]{BUlowerdim}}]\label{butht}
Let $X$ be an $m$-manifold with a free $\Z_2$ action $\tau$.
Then $(X, \tau, m)$ is a BU-triple if and only if $\ind(X) = \mathrm{ht}(X) = m$.
\end{proposition}
\begin{corollary}\label{bueqcor}
Let $X$ be an $m$-manifold with a free $\Z_2$ action $\tau$.
If $\hts(X) = m-1$ then $\ind(X)$ is also $m-1$ and $(X, \tau, m-1)$ is a BU-triple. Moreover, if $\mathrm{ht}(X) < m$, then $\mathrm{ind}(X) < m$.
\end{corollary}

We now turn to spaces that are of central interest in the present article. 
As introduced in \Cref{sec:intro}, let $\alpha$ be a generic $n$-length vector.
Let $\mathrm{M}_{\alpha}$ be the corresponding moduli space of planar polygons whose side lengths are specified by $\alpha$.
Observe that $\mathrm{M}_{\alpha}$ admits an involution $\tau$ defined by
\begin{equation}\label{invo}
    \tau(v_{1},v_{2},\dots,v_{n})=(\bar{v}_{1},\bar{v}_{2},\dots,\bar{v}_{n}),
\end{equation}
where $\bar{v}_{i}=(x_{i},-y_{i})$ and $v_{i}=(x_{i},y_{i})$. 
This involution sends a polygon to its reflected image across the $X$-axis. 
Since $\alpha$ is generic, $\tau$ is a fixed point free involution.

Since the diffeomorphism type of a moduli space does not depend on the ordering of the side lengths of polygons, we assume that  $\alpha=(\alpha_{1},\alpha_{2},\dots,\alpha_{n})$ satisfies $\alpha_{1}\leq \alpha_{2} \leq \dots\leq \alpha_{n}$. 

\begin{definition}
\emph{Given a length vector $\alpha$, a subset $I\subset [n]$ is called $\alpha$-\emph{short}  if 
\[\sum_{i\in I} \alpha_i  < \sum_{j \not \in I} \alpha_j\]  and \emph{long} otherwise.}

\end{definition}	


\begin{definition}
\emph{For a length vector $\alpha$, consider the collection of subsets of $[n]$:
\[ S_{n}(\alpha) = \{J\subset [n] : \text{ $n\in J$ and $J$ is short}\}, \]
and a partial order $\leq$ on $S_{n}(\alpha)$ by $I\leq J$ if $I=\{i_{1},\dots,i_{t}\}$ and $\{j_{1},\dots,j_{t}\}\subseteq J$ with $i_{s}\leq j_{s}$ for $1\leq s\leq t$. 
The \emph{genetic code} of $\alpha$ is the set of maximal elements of $S_{n}(\alpha)$ with respect to this partial order. 
If $A_{1}, A_{2},\dots, A_{k}$  are the maximal elements of $S_{n}(\alpha)$ with respect to $\leq$ then the genetic code of $\alpha$ is denoted by $\langle A_{1},\dots,A_{k} \rangle$.}
\end{definition}
	
\begin{remark}
Let the $n$-tuple $\alpha=(1,\dots,1,n-2)$ be a length vector. 
Then the genetic code of $\alpha$ is $\langle n \rangle$. 
It follows from \cite[Example 6.5]{HR1} that $\mathrm{M}_{\alpha}\cong S^{n-3}$ and $\overline{\mathrm{M}}_{\alpha}\cong \mathbb{R}P^{n-3}$.
\end{remark}
	
If $S$ is a gene, then $S \setminus \{n\}$ is a {\it gee}.
We set $[n-1]=\{1,2,\dots,n-1\}$.
An $S \subset [n-1]$ is a {\it subgee} if $S \le T$ for some gee $T$.
Many topological invariants of these moduli spaces have been expressed in terms of the genetic code. 
We state here a result describing the cohomological algebra of $H^*(\mbar_{\alpha};\Z_2)$ in terms of the genetic information.

\begin{theorem} [{\cite[Corollary 9.2]{cohomologyring}}] \label{hausmann-knutson}
The cohomology ring $H^{\ast}(\mbar_{\alpha}; \Z_2)$ is generated by classes $R, V_1,$ $ V_2, \dots, V_n \in H^1(\mbar_{\alpha}; \Z_2)$ subject to the following relations:
\begin{enumerate}
\item[(R1)] $R V_i$ + $V_i^2$, for $i \in [n]$.
\item[(R2)] $V_S := \displaystyle\prod_{i \in S}V_i$, unless $S$ is a subgee.
\item[(R3)] For every subgee $S$ with $|S| \ge n-d-2$,
\[\displaystyle\sum_{T\cap S=\emptyset} R^{d - |T|}V_T,\] where $T$ is a subgee. 
\end{enumerate}
\end{theorem}

\section{Length vectors with monogenic codes}\label{sec3}

A genetic code with only one gene is called {\it monogenic}.
This section deals with computations of coindex, index and height of certain planar polygon spaces having monogenic code. 

\begin{proposition}\label{prop:M(l) is sphere}
Let $\langle \{n\} \rangle$ be the genetic code of $n$-length vector $\alpha$. Then $\malpha$ is tidy.
\end{proposition}
\begin{proof}
For a length vector with genetic code $\langle \{n\} \rangle$, it is well known that $\malpha$ is homeomorphic to $S^{n-3}$. The identity map on $S^{n-3}$ implies that the coindex of $\malpha$ is at least $n-3$. By Borsuk-Ulam theorem, there does not exist a $\Z_2$-equivariant map from $S^{n-2} \to S^{n-3}$. Therefore, $\mathrm{coind}(\malpha) = n-3 \le \mathrm{ind}(\malpha) \le \mathrm{dim}(\malpha ) = n-3$. Hence $\malpha$ is tidy.
\end{proof}

\begin{lemma}\label{lem:lower bound on coind}
If the size of the smallest gee in the genetic code corresponding to an $n$-length vector $\alpha$ is $l$, then $ \emph{coind}(\malpha) \ge n-3-l$.
\end{lemma}
\begin{proof}
Let $\alpha=(\alpha_1,\dots,\alpha_n)$ and $S\subset [n]$ be the smallest gee of the genetic code of $\alpha$. 
Consider the shorter length vector \[\alpha(S)=(\alpha_{i_1},\dots,\alpha_{i_k},\alpha_n+\sum_{j\in S}\alpha_j),\] where $S^c=\{i_1,\dots,i_k,n\}$.
Since $S\cup \{n\}$ is a gene of the genetic code of $\alpha$, the set $S\cup \{n\}\cup \{i_s\}$ is long for each $1\leq s\leq k$.
Therefore, $\mathrm{M}_{\alpha(S)}\cong S^{n-l-3}$ as $\alpha(S)$ is a $(n-l-3)$-length vector.
Note that any polygon with side lengths given by length vector $\alpha(S)$ can be considered as a polygon with side lengths given by $\alpha$ whose sides indexed by $S$ are parallel.
This gives a $\Z_2$-equivariant embedding of $S^{n-l-3}$ in $\mathrm{M}_{\alpha}$.
This proves the lemma.
\end{proof} 	

The above result will be used throughout this article. 
In particular, this helps in the case of length vector $\alpha$ with a gene of size $2$.  
First we collect two results we will need. 

\begin{lemma} \label{pinz}
Let $X$ be a $\mathbb{Z}_2$-space such that $\ind(X) = \co(X) = n$ for $n\geq 1$.
Then the homotopy group $\pi_n(X)$ cannot be a torsion group. 
In fact, $\pi_n(X)$ has an infinite cyclic quotient. 
\end{lemma}
\begin{proof}
Since the index and the coindex of $X$ are the same, we have 
\[S^n \to X \to S^n \]
such that both the maps are $\Z_2$ equivariant. 
The composition gives us a self-map which is antipodal-preserving. 
Recall that such a map has odd degree. 
Therefore the induced map on $\pi_n(S^n)$ is a non-trivial homomorphism with infinite cyclic image that factors through $\pi_n(X)$. 
\end{proof}

We now consider the homotopy type of the universal cover of some polygon spaces. 
It is proved in \cite[Theorem B]{kamiyama10} that for the genetic code $\langle\{n-1, n\}\rangle$ then the the universal cover of the corresponding polygon space is homotopic to wedge of infinitely many spheres.
The first step in the proof is to realize that the moduli space is homeomorphic to the connected sum of $n-1$ copies of $\S^1\times \S^{n-4}$. Then, one uses the topological Bass-Serre theory to construct the universal cover. 
Finally, one realizes that the universal cover is exhausted by wedge of $(n-4)$-spheres.
We note that even though the statement of \cite[Theorem B]{kamiyama10} mentions only one genetic code the proof applies to all genetic codes of the type $\langle\{b, n\}\rangle$.

\begin{lemma}\label{unicovm}
Let $\langle \{b,n\}\rangle$ be the genetic code of a length vector $\alpha$ and $n\geq 6$.
Then the universal cover of $\malpha$ has the homotopy type of wedge of infinitely many spheres of dimension $n-4$.
\end{lemma}

\begin{proof}
Note that as a consequence of \cite[Proposition 2.10]{geohausmann} the moduli space corresponding to the genetic code $\langle\{b,n\}\rangle$ is homeomorphic to the connected sum  $\sharp_{b} (\S^1 \times \S^{m-1})$. 

The universal cover of a connected sum (of sphere products, in particular) can be constructed using topological Bass-Serre theory. We refer the reader to \cite{gadgil, mcc81, topgroup},   for more details. 

We sketch the construction here. 
Without loss of generality we assume that there are only two summands. 
It is useful to consider the connected sum as a handlebody, i.e., first take a copy of $\S^{n-3}$ minus the interior of four disjoint embedded balls, say $C_1, C_2, D_1, D_2$, denote this space by $P$. 
Then attach a handle connecting boundaries of $C_i$'s with $D_i$'s to obtain a space homeomorphic to $M := (\S^1\times \S^{n-4})\sharp (\S^1\times \S^{n-4})$. 
Note that $\pi_1(M)$ is the free group of rank $2$ with generators $\alpha_1, \alpha_2$. 
In order to construct the universal cover $\tilde{M}$ take a copy of $P$ for each element $g$ in the fundamental group and call it $gP$. 
Finally, attach handles between $C_i$ of $gP$ and $D_i$ of $g\alpha_i P$ for each $1\leq i\leq 2$. 

Observe that if $S$ is a finite subset of $\pi_1(M)$, then the union of $gP$ for $g\in S$ is a compact, simply-connected space that gives an exhaustion of $\tilde{M}$. 
This space is homeomorphic to $\S^{n-3}$ minus finitely many disjoint embedded balls; hence it is homotopy equivalent to a wedge of finitely many $\S^{n-4}$'s. 
The homology of $\tilde{M}$ is the direct limit of homologies of wedge of spheres which is isomorphic to the homology of wedge of infinitely many spheres. 
Since both the spaces are simply connected this homology isomorphism is indeed a homotopy equivalence. 
\end{proof}

\begin{proposition}\label{gcan}
Let $\langle \{b,n\}\rangle$ be the genetic code of a length vector $\alpha$ and $n\neq 6$. Then
$\ind(\malpha)=n-4$ if $b$ is odd, and $n-3$ if $b$ is even, while $\co(\malpha)=n-4$ always. Therefore $\malpha$ is tidy iff $b$ is odd.

\end{proposition}
\begin{proof}
Since the genetic code is $\langle  \{b,n\} \rangle$, the collection of gees is $\{\{b\}\}$, and hence \[\{\emptyset, \{1\},\{2\},\dots, \{b\}  \}\] is the set of all the subgees.
For $i \ne j$, the relation (R2) of \Cref{hausmann-knutson} implies that $V_iV_j = 0$ if $i \ne j$.
Let $m= n-3$ and $d \ge m$. Consider the subgee $S = \{1\} $, then subgees disjoint from $S$ are $\emptyset, \{2\}, \{3\}, \dots, \{b\}$.
Using the relation (R3) of \Cref{hausmann-knutson}, we get \[R^d + \displaystyle\sum_{i=2}^{b} R^{d -1}V_i = 0.\]
For the subgee $S = \{j\},\ 1 \le j \le b$, above equation gives: 
\begin{equation}\label{eq:subjee j}
R^d + \displaystyle\sum_{j \ne i=1}^{b} R^{d -1}V_i = 0.
\end{equation}
Comparing these relations with each other, we get that  $RV_1 = RV_2 =RV_3= \dots =RV_b$. Thus, \Cref{eq:subjee j} implies that $R^d = \displaystyle\sum_{i=2}^{b} R^{d -1}V_1 = (b-1)R^{d -1}V_1.$
Hence, for $d=m$, we have 
\begin{equation} \label{eq:subjee 1}
R^m = (b-1)R^{m -1}V_1.
\end{equation}
By \Cref{lem:lower bound on coind}, we see that $\co (\malpha) \ge m-1$. From \Cref{eq: coind ht ind}, we get  \[ m-1 \le \co (\malpha) \le \hts (\malpha) \le  \ind (\malpha) \le \dim(\malpha) = m.\]

We now analyse \Cref{eq:subjee 1} further. 
Since $R^{m-1} V_1$ spans $H^m(\malpha, \Z_2)$ we have $R^m = 0 $ if and only if $(b-1) \equiv 0$ (mod 2) implying that $b$ is odd. 
Thus for odd values of $b$, the Stiefel-Whitney height $\mathrm{ht}(\malpha) = m-1$. 
By \cref{butht} we have that $(\malpha, \tau, m)$ is not a BU-triple and hence, $\ind(\malpha) = m-1$.  
Hence $\malpha$ is tidy whenever $b$ is an odd integer.

Now let $b$ be even. 
Then we get that $\ind(\malpha) = \hts(\malpha) = m$. 
Therefore to find whether $\malpha$ is tidy or not, it suffices to find if $\co(\malpha)$ is $m-1$ or $m$. We first observe that the space $\malpha \cong \sharp_{b} (\S^1 \times \S^{m-1})$. 
If $m=2$, then $\malpha \cong \sharp_{b} (\S^1 \times \S^{1})$ and hence is not tidy (see \cite[Section 5.3, page 100]{UBUthm}). 

Let $m\geq 4$ and $U(\malpha)$ denote the universal cover of $\malpha$ which has the homotopy type of wedge of infinitely many copies of $\S^{m-1}$'s (\cref{unicovm}). 
Let $\Lambda$ denote the set indexing the spheres in the wedge product. 
Since the homotopy group of a wedge sum of spheres is a colimit of the homotopy groups in that dimension taken over all finite subsets, we have 
    \begin{align*}
				\pi_m(U(\malpha))
				&= \pi_m(\displaystyle\vee_{\Lambda} \S^{m-1})\\ 
				&= \colim_{\text{fin } F \subset \Lambda} \pi_m (\displaystyle\vee_{F} \S^{m-1})\\
				& = \colim_{\text{fin } F \subset \Lambda} \displaystyle\oplus_{F} (\pi_m (\S^{m-1})) \\
				& = \colim_{\text{fin } F \subset \Lambda} \displaystyle\oplus_{F} (\Z_2).
	\end{align*} 

Hence $\pi_m(\malpha) = \pi_m(U(\malpha))$ is a torsion group and consequently $\malpha$ can't be tidy (\cref{pinz}). 
Therefore, $m-1 = \co(\malpha) < \ind(\malpha) = m$. 
Hence $\malpha$ is non-tidy for even values of $b$, and tidy for odd values of $b$. 
\end{proof}

\begin{remark}
We note two observations here. 
\begin{enumerate}
\item Let $\Sigma_g$ be the orientable surface of genus $g$. Recall that, for $1\leq g\leq4$ if the genetic code of a length vector $\alpha(g)$ is $\langle \{g,5\}\rangle$ then $\mathrm{M}_{\alpha(g)}\cong \Sigma_g$. Therefore, $\mathrm{M}_{\alpha(g)}$ is tidy if and only if $g=1,3$.
\item Let $\alpha = (1,1,\dots, 1, n-4)$ be an $n$-length vector, then the only short subsets containing $n$ are  $\{n\}$ and $\{n-1,n\}$. 
Therefore the genetic code in this case is $\langle\{n-1,n\} \rangle$. By \Cref{gcan}, the corresponding $\mathrm{M}_{\alpha}$ is tidy if $n$ is even and non-tidy otherwise. 
\end{enumerate}
\end{remark}

Consider the following genetic codes: 
\begin{enumerate}
\item $G_1=\langle \{1,\dots,n-4,n\} \rangle$,
\item $G_2=\langle \{1,\dots,n-5,n-3,n\} \rangle$,
\item $G_3=\langle \{1,\dots,n-5,n-2,n\} \rangle$,
\item $G_4=\langle \{1,\dots,n-5,n-1,n\} \rangle$.
\end{enumerate}
Let $\mathrm{M}_{G_i}$ be the planar polygon space corresponding to $G_i$. 
It follows from \cite[Proposition 2.1]{geohausmann} that  $\mathrm{M}_{G_i}\cong (S^1)^{n-2}\times \mathrm{M}_{\alpha(i)}$ where the genetic code of $\alpha(i)$ is $\langle\{ i,5\}\rangle$ for $1\leq i \leq 4$. 
The free $\Z_2$-action on $\mathrm{M}_{G_i}$ is given by an \Cref{invo}. 

With the above notations, we have the following proposition.
\begin{proposition}\label{long gcodes}
For $1\leq i \leq 4$ the space $\mathrm{M}_{G_i}$ is tidy if and only if $i=1, 3$. 
\end{proposition}
\begin{proof}
Note that for each $1\leq i\leq 4$, the projection  $\mathrm{M}_{G_i}\to \mathrm{M}_{\alpha(i)}$ is a $\Z_2$-map. 
Therefore, $\ind(\mathrm{M}_{G_i})\leq \ind(\mathrm{M}_{\alpha(i)})$.
Since $\mathrm{M}_{\alpha(i)}\cong \Sigma_i$, $\ind(\mathrm{M}_{G_i})\leq 2$ for $1\leq i\leq 4$. It is easy to see that
\[\ind(\Sigma_i) = \begin{cases}
1 & \text{if $i$ is an odd integer, }\\[10pt]
2 & \text{if $i$ is an even integer.}
\end{cases}\]
Therefore, 
\[\ind(\mathrm{M}_{G_i}) \leq \begin{cases}
1 & \text{if $i=1,3$, }\\[10pt]
2 & \text{if $i=2,4$.}
\end{cases}\]
Let $\alpha(G_i)$ be the length vector corresponding to the genetic code $G_i$. 
Note that $J=\{1,\dots, n-4\}$ is a short subset with respect to each genetic code $G_i$. 
Therefore, we can reduce $\alpha(G_i)$ to the length vector \[\alpha(G_i,J)=(\alpha_{m-3},\alpha_{m-2},\alpha_{m-1},\alpha_{n}+\sum_{i=1}^{n-4}\alpha_{i}).\] 
Observe that, $\mathrm{M}_{\alpha(G_i,J)}\cong S^{1}$. Therefore, we have an $\Z_2$-equivariant embedding of $\mathrm{M}_{\alpha(G_i,J)}$ in $\mathrm{M}_{G_i}$. 
This gives \[1\leq \co(\mathrm{M}_{G_i}),  \textit{ for } 1\leq i\leq 4.\] 
Now it is clear that $\mathrm{M}_{G_i}$ is tidy if $i=1,3$. 

For $i=2,4$ we now prove that $\co(\mathrm{M}_{G_i})=1$ and $\ind(\mathrm{M}_{G_i})=2$. 
It follows from \cite[Theorem 2.3]{mfdav} that the $2\leq \hts(\mathrm{M}_{G_i}))$ for $i=2,4$.
Therefore, \[\ind(\mathrm{M}_{G_i})=2, \text{ for } i=2,4.\] 
Now it follows from 
\Cref{gcan} that \[\co(\mathrm{M}_{G_i})=1, \text{ for } i=2,4.\] 
Therefore, $\mathrm{M}_{G_i}$ is not tidy for $i=2,4$. This proves the proposition.
\end{proof}


\section{Formula for the Stiefel-Whitney height}\label{swheight}

In this section, we deal with the $n$-length vectors that correspond to the genetic code having two genes with one gene of size 2.
We begin with obtaining a formula for $R^{n-3}$,  when the genetic code is $\langle \{g_1,\dots,g_k,n\}, \{b,n\}\rangle$, where $g_1 < g_2 < \dots < g_k < b$. This formula is a generalization of Davis' formula  \cite[Theorem 1.4]{Davisformula}.

Let us fix the following notations for the sake of simplicity of writing.
\begin{enumerate}
    \item Let $\Z_{\geq 0}$ be the set of non-negative integers, and
    \[S_k=\{(b_1,\dots,b_k)\in \Z_{\geq 0}^{k} : \sum_{j=0}^{i-1}b_{k-j}\leq i \text{ for } 1\leq i \leq k \}.\]
    \item Let $B=(b_1,\dots,b_k)\in \Z_{\geq 0}^{k}$. Denote $B_{+}=\sum_{i=1}^{k}b_{i}$.
    \item Let $J=\{j_1,\dots,j_r \}$ be a set of distinct positive integers such that  $j_i\leq g_k$ for $1\leq i\leq r$.
    Let $\theta(J)=(\theta_1,\dots,\theta_k)$, where \[\theta_i=|\{j\in J: g_{i-1}< j\leq g_i\}|.\] Throughout this article, $|A|$ denotes the cardinality of the set $ A$, and we take $g_0=0$.
\end{enumerate}
It is easy to see that $J$ is a subgee dominated by $\{g_1, \dots, g_k\}$ if and only if $\theta(J)\in S_k$.
With the above notations, Davis proved the following result.

\begin{theorem}[{\cite[Theorem 1.4]{Davisformula}}\label{constant}]
Let $\langle \{g_1,\dots,g_k,n\}\rangle$ be the genetic code of $\alpha$ and $a_i=g_{k+1-i}-g_{k-i}$ for $1\leq i \leq k$.
Let $\phi: H^{n-3}(\mbalpha;\Z_2)\to \Z_2$ be the Poincare-duality isomorphism and $J$ be a subgee of cardinality $r$. Then
\[\phi(R^{n-3-r}V_{J})=\sum_{B}\prod_{i=1}^{k}\binom{a_i+b_i-2}{b_i},\]
where \[B\in\{(b_1,\dots,b_k)\in \Z_{\ge 0}^k: \sum_{i=1}^{k} b_i=k-r, \ (b_1,\dots,b_k)+\theta(J)\in S_k\}.\]
\end{theorem}

We extend the above result for the genetic code $\langle \{g_1,\dots,g_k,n\}, \{b,n\}\rangle$ with $g_1 <\dots <g_k < b$.
\begin{theorem}\label{powerR}
Let $\langle \{g_1,\dots,g_k,n\},\{b,n\}\rangle$ be the genetic code of $\alpha$ and $a_i=g_{k+1-i}-g_{k-i}$ for $1\leq i \leq k$. 
Then \[R^{n-3}=\sum_{B}\prod_{i=1}^{k}\binom{a_i+b_i-2}{b_i}+(b-g_k),\] where $B\in S_k$ with $B_{+}=k$. 
\end{theorem}

To prove this theorem, we need some further notations and some more results which we derive now.
\begin{enumerate}
    \item The collection $\mathcal{S}_i = \{P \subseteq [n-1]:  P\text{ is a subgee and } |P| = i\}$.
    \item For a subgee $P$, its support set is $\mathrm{supp}(P) = \{s\in [n-1] : s\notin P \text{ and  $P\cup\{s\}$ a subgee}  \}.$
    \item For a subgee $P$ and $s\in \mathrm{supp}(P)$, $\mathcal{S}_i(P,s)=\{S\in \mathcal{S}_i : P\subseteq S, ~ s\notin S \}$.    
    \item For $1\leq i\leq k$, the set $(g_{i-1},g_i]:=\{j\in \Z_{\geq 0} : g_{i-1}<j\leq g_i\}$ and $(g_k,b]$ is defined similarly. These sets are referred to as {\it blocks. }
\end{enumerate}
\begin{proposition}\label{StPs}
Let $P$ be a subgee and $\{s,s'\}\subseteq \mathrm{supp}(P)$. If $s$ and $s'$ are from  the same block, then for each $t\in [k]$ \[| \mathcal{S}_t(P,s)| = |\mathcal{S}_t(P,s')|.\]
\end{proposition}
\begin{proof}
Let $t \in [k]$ be fixed. Define $f:\mathcal{S}_t(P,s)\to  \mathcal{S}_t(P,s')$ by \[f(S) = \begin{cases}
S & \text{if $s'\notin S$, }\\[10pt]
(S\setminus \{\ s'\})\cup \{s\} & \text{if $s'\in S$.}
\end{cases}\]
It is easy to see that $f$ is a bijection.
\end{proof}

The next result shows that the conclusion of \Cref{StPs} holds even if $s, s' \in  \mathrm{supp} (P)$ and are not in the same block, provided $|P| = k-1$.

\begin{proposition}\label{SkPs}
If $P\in \mathcal{S}_{k-1}$ and $\{s,s'\}\subseteq \mathrm{supp}(P)$, then \[ |\mathcal{S}_k(P,s) | = | \mathcal{S}_k(P,s')|.\]
\end{proposition}
\begin{proof}
We observe that both $P\cup \{s\}$ and $P\cup \{s'\}$ are subgees of cardinality $k$, and thus
\[\mathcal{S}_k(P,s)\setminus (P\cup \{s'\})=\mathcal{S}_k(P,s')\setminus (P\cup \{s\}).\] This proves the result.
\end{proof}

In \cite[Corollary 1.6]{cohomologyclasses}, the author proved that for the monogenic code $\langle \{g_1,g_2,\dots,g_k,n\}\rangle$ the image of $R^{m-|J|}V_J$ under $\phi$  (the Poincare duality isomorphism) is $1$  whenever the subgee $J$ is of cardinality $k$. Observe that the same proof works for any genetic code whenever $J$ is of maximum cardinality. Here we reproduce the proof in our case for completeness.

\begin{theorem}\label{thm:valueofmaxunderPD}
Let $\langle \{g_1,\dots,g_k,n\},\{b,n\}\rangle$ be the genetic code of a length vector $\alpha$ and $J$ be a subgee of cardinality $k$. Then $\phi(R^{m-|J|}V_J) =1$.
\end{theorem}
\begin{proof}
Note that the product $V_J$ is non-zero since it is non-zero in the quotient $H^*(\mbalpha; \Z_2)/\langle R\rangle$, which is an exterior face ring (see \cite[Corollary 1.6]{cohomologyclasses}).
By Poincar\'e duality, there must be an  $X=\sum_{i} X_i\in H^{m-|J|}(\mbalpha; \mathbb{Z}_2)$ such that $X\cdot V_J=1$ and thus $\phi(X\cdot V_J)=1$. 
Since $J$ is maximal, if $X_i$ contains a factor $V_t$ such that $t\notin J$, then $X_i\cdot V_J=0$.  Now using the relation $(R1)$ of \Cref{hausmann-knutson}, we can replace $V_t^2$ by $RV_t$, whenever $t\in J$. 
Therefore if $X_i\cdot V_J\neq 0$, then each $X_i$ can be replaced by $R^{m-|J|}$.  Since $\phi(X\cdot V_J)=1$, number of such $X_i$'s must be odd. This proves the theorem.
\end{proof}

Let $m=n-3$ be the dimension of $\malpha$. For a generic length vector $\alpha=(\alpha_1,\dots,\alpha_n)$, the following result is an important tool to compute $R^{m}$.
\begin{lemma}\label{implemma}
Let $P$ be a nonempty subgee and $s\in\mathrm{supp}(P)$. Then \[R^{m-|P|} V_P=\sum_{|P|+1\leq t\leq k}\left(\sum_{S \in \cs_t(P,s)}R^{m-|S|}V_S \right).\]

\end{lemma}


Before we prove \Cref{implemma}, we make a few observations which play a critical role in its proof.
For a nonempty subgee $P$, let $E_P$ denotes the equation obtained from the relation $(R3)$ of \Cref{hausmann-knutson}, {\it i.e.}, 
\begin{equation}\label{EqR3}
    E_P: \sum_{T\cap P=\emptyset} R^{m - |T|}V_T=0.
\end{equation}
\begin{remark}\label{rem:obs1}
    If $P\neq \emptyset$ and $s\in \mathrm{supp}(P)$, then  adding $E_P$ and $E_{P\cup \{s\}}$, we get the following equation
    \begin{equation}
        E_P+E_{P\cup \{s\}}: \sum_{T\cap P=\emptyset,~ s\in T} R^{m - |T|}V_T=0.
    \end{equation}

\end{remark}

\begin{remark}\label{rem:obs2}
    Let $|P|\geq 2$, $s'\in P$ and $s\in \mathrm{supp}(P)$. 
    Then adding $E_P+E_{P\cup \{s\}}$ and $E_{P\setminus \{s'\}}+E_{(P\setminus \{s'\})\cup \{s\}}$, we get the following 
    \begin{equation}
        R^{m-2}V_{\{s,s'\}}=\sum_{\substack{T\cap (P\setminus\{s'\})=\emptyset\\ \{s,s'\}\subsetneq T}} R^{m - |T|}V_T.
    \end{equation}
\end{remark}

\begin{proof}[Proof of \Cref{implemma}]
    We prove this using a binary tree representation. Given a subgee $P = \{i_1, i_2, \dots, i_k\}$ and $s\in \mathrm{supp}(P)$, we construct a binary tree $B_k$ of height $k$ as follows: For $r \in \{0, 1, \dots k-1\}$, all the vertices of $B_k$ at depth $r$ are named as ${i_r}$. To label the vertices at depth $k$, {\it i.e.} the leaf vertices, we proceed as follows: Since $B_k$ is a binary tree, each non-leaf vertex has two children, {\it i.e.}, a left child and a right child. For any leaf vertex, there is a unique path that connects the root vertex to this leaf vertex. Consider this path, and take the set $\{s\}$. We start moving on this path starting from the root vertex. If on the path, we move to the left child from a vertex $v$, then we include $v$ in the set; if we move to the right child, then we do not include  the vertex  $v$ (see \Cref{fig:Bk} for the illustration when $k=3$.).
    
   	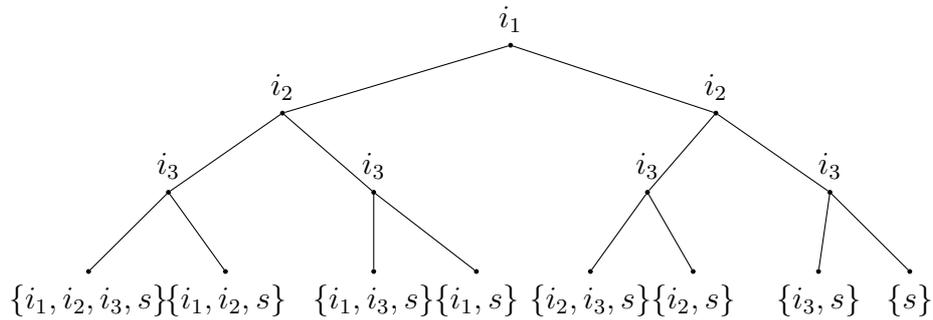
\begin{figure}[H]
	\centering
		\begin{tikzpicture}
			[scale=0.30, vertices/.style={draw, fill=black, circle, inner
				sep=0.5pt}]

			\node[vertices, label=above:{{$i_1$}}] (i1) at (2,10) {};
			
			\node[vertices,label=above:{{$i_2$}}] (i21) at (-8, 7) {};
			\node[vertices,label=above:{{$i_2$}}] (i22) at (11, 7) {};
			
			\node[vertices,label=above:{{$i_3$}}] (i31) at (-13, 3.5) {};
			\node[vertices,label=above:{{$i_3$}}] (i32) at (-4, 3.5) {};
			\node[vertices,label=above:{{$i_3$}}] (i33) at (8, 3.5) {};
			\node[vertices,label=above:{{$i_3$}}] (i34) at (16, 3.5) {};

			\node[vertices,label=below:{{$\{i_1, i_2, i_3,s \}$}}] (i311) at (-16.5, 0) {};
			
		    \node[vertices,label=below:{{$\{i_1, i_2 ,s\}$}}] (i312) at (-10.5, 0) {};
		    
		    \node[vertices,label=below:{{$\{i_1, i_3 ,s\} $}}] (i321) at (-4, 0) {};

		    \node[vertices,label=below:{{$\{i_1,s\} $}}] (i322) at (0.5, 0) {};
		    
		    \node[vertices,label=below:{{$\{i_2, i_3,s\}$}}] (i331) at (5.5, 0) {};

		    \node[vertices,label=below:{{$\{i_2,s\}$}}] (i332) at (10, 0) {};

			\node[vertices,label=below:{{$\{i_3,s\}$}}] (i341) at (15.5, 0) {};

			\node[vertices,label=below:{{$\{s\}$}}] (i342) at (19.5, 0) {};

			\foreach \to/\from in
			{i1/i21, i1/i22, i21/i31, i21/i32, i22/i33, i22/i34, i31/i311, i31/i312, i32/i321, i32/i322, i33/i331, i33/i332, i34/i341, i34/i342} \draw [-] (\to)--(\from);
			
			\end{tikzpicture}
			\vspace{0.5cm}
			\caption{Tree $B_3$} \label{fig:Bk}

    \end{figure}
    
   Since $s \in \mathrm{supp}(P)$, $P \cup \{s\}$ is a subgee, and hence every subset of it is also a subgee.
   We now associate an equation to every vertex of $B_k$ starting from the leaf vertices. If a leaf vertex corresponds to a subset, say $A$ of $P \cup \{s\}$, then the equation associated with the leaf $A$ is $E_A$ (see \Cref{EqR3}).
    We now associate equations to the parents of leaves by adding the equation associated with their children. We do this recursively in the decreasing order of the depth of the vertices by using \Cref{rem:obs1} and \Cref{rem:obs2}. Finally, we see that the equation corresponding to the root vertex is
\[R^{m-|P|} V_P=\sum_{|P|+1\leq t\leq k}\left(\sum_{S \in \cs_t(P, s)}R^{m-|S|}V_S \right).\]   
\end{proof}


\begin{corollary}\label{Rcorr}
If $P\in \mathcal{S}_{k-1}$ and $s,s'\in \mathrm{supp}(P)$, then $R^{m-k}V_{P\cup \{s\}}=R^{m-k}V_{P\cup \{s'\}}$.
\end{corollary}
\begin{proof}
The proof follows from \Cref{implemma} and the following observation \[\mathcal{S}_k(P,s)\setminus (P\cup \{s'\})=\mathcal{S}_k(P,s')\setminus (P\cup \{s\}).\] 
\end{proof}

\begin{proposition}\label{RViRVj}
If $S,S'\in \mathcal{S}_k$, then $R^{m-k}V_S=R^{m-k}V_{S'}$.
\end{proposition}
\begin{proof}
Let $T$ be the gee $\{g_1,\dots,g_k\}$. We will show that for any $S\in \mathcal{S}_k$, $R^{m-k}V_S=R^{m-k}V_T$. Let $S=\{i_1,\dots,i_k\}$ be a subgee of size $k$, where $i_1<\dots<i_k$. 
It follows from Corollary \ref{Rcorr} that $R^{m-k}V_T=R^{m-k}V_{\{g_1,\dots,g_{k-1},i_k\}}.$
Similar argument implies that  $R^{m-k}V_{\{g_1,\dots,g_{k-1},i_k\}}=R^{m-k}V_{\{g_1,\dots,g_{k-2},i_{k-1}, i_k\}}.$
By continuing this way, we get \[R^{m-k}V_{\{g_1,i_2,\dots,i_k\}}=R^{m-k}V_S.\]
Therefore, for any two subgees of size $k$, the above equation holds thereby proving the result.
\end{proof}

\begin{lemma}\label{lemmasubgee}
Let $P$ be a subgee with $|P|=r$ and $s,s'\in \mathrm{supp}(P)$. If $r\leq k-1$, then \[R^{m-r-1}V_{P\cup \{s\}}=R^{m-r-1}V_{P\cup \{s'\}},\] whenever $s$ and $s'$ are in the same block.
\end{lemma}
\begin{proof}
We prove this in decreasing order of $r$. From  Corollary \ref{Rcorr}, the result holds for $r=k-1$. 
Let $|P|=r$ and result holds for all subgees of cardinality bigger than $r$. Let $s,s'\in \mathrm{supp}(P)$.  Using \Cref{implemma}, we get  

\[\sum_{r+1\leq i\leq k}\left(\sum_{S\in \mathcal{S}_i(P,s)}R^{m-i}V_S \right)=\sum_{r+1\leq j\leq k}\left(\sum_{S\in \mathcal{S}_j(P,s')}R^{m-j}V_S \right).\]
Using induction and \Cref{StPs}, we get  \[\sum_{S\in \mathcal{S}_{r+1}(P,s)}R^{m-t-1}V_S=\sum_{S\in \mathcal{S}_{r+1}(P,s')}R^{m-t-1}V_S.\]
The result now follows from \Cref{SkPs}.
\end{proof}



\begin{proposition}
Let $G=\langle\{g_1,\dots,g_k\},\{b,n\}\rangle$ and $G'=\langle\{g_1,\dots,g_k\}\rangle$ be the genetic codes of $\alpha$ and $\beta$, respectively.
Let $P$ be a nonempty subgee such that ${supp}(P)\neq \emptyset$. 
Then the value of $R^{m-|P|}V_{P}$ under the corresponding Poincare duality isomorphisms from $H^{m}(\mbalpha, \Z_2)$ and $H^{m}(\overline{\mathrm{M}}_{\beta},\Z_2)$ remains same.
\end{proposition}

\begin{proof}
Observe that if $|P|=k$, then the result follows from \Cref{thm:valueofmaxunderPD}. 
Now we assume that $|P|<k$. In this case, $P$ is dominated by $\{g_1,\dots,g_k\}$.
Note that the collections of subgees of cardinality greater than $1$ with respect to both $G$ and $G'$ coincides. Moreover, the equations in \Cref{implemma} also remain unchanged for both the genetic codes. Therefore, \Cref{thm:valueofmaxunderPD} implies the result.
\end{proof}

The following is now an immediate consequence of the above result and \Cref{constant}.

\begin{corollary}\label{corollary:constant}
Let $\langle \{g_1,\dots,g_k,n\},\{b,n\}\rangle$ be the genetic code of $\alpha$ and $a_i=g_{k+1-i}-g_{k-i}$ for $1\leq i \leq k$.
Let $P$ be a subgee of cardinality $r\in \{1,\dots,k-1\}$ dominated by $\{g_1,\dots,g_k\}$. Then
\[\phi(R^{m-r}V_{P})=\sum_{B}\prod_{i=1}^{k}\binom{a_i+b_i-2}{b_i},\]
where \[B\in\{(b_1,\dots,b_k)\in \Z_{\ge 0}^k: \sum_{i=1}^{k} b_i=k-r, \textit{ and } (b_1,\dots,b_k)+\theta(P)\in S_k\}.\]
\end{corollary}

We now have all the necessary machinery for \Cref{powerR}. So we next see its proof.

\begin{proof}[Proof of \Cref{powerR}]
Without loss of generality, we omit the powers of $R$ just to simplify the notations.  
Comparing $E_{\{g_k\}}$ with $E_{\{b\}}$ we get.
\begin{equation}\label{eqA}
V_b=V_{g_k}+\sum_{\substack{|S|\geq 2\\ g_k\in S}}V_S.    
\end{equation}
Let $P=\{g_k\}$ and $s=g_{k-1}$. 
Then \Cref{implemma} gives \[V_{g_k}=\sum_{\substack{|S|\geq 2\\ g_k\in S,~g_{k-1}\notin S}} V_S.\]
Comparing it with \Cref{eqA}, we get 
\begin{equation}\label{eqB}
V_b=V_{\{g_k,g_{k-1}\}}+ \sum_{\substack{|S|\geq 3\\ \{g_k,g_{k-1}\}\subseteq S}}V_S.
\end{equation} 
Recall that $a_j=g_{k+1-j}-g_{k-j}$ for $1\leq j \leq k$. 
To get a more explicit expression for the monomials $V_S$, in the above equation, we note the following. 
Any subgee $S$ of cardinality at least $3$ and containing $\{g_k, g_{k-1}\}$ corresponds to a $k$-typle $C:= (c_1,\dots,c_k)$ such that $C + (1,1,0,\dots,0)$ belongs to $S_k$. 
In order to expand each $V_S$ as a product of generating classes we note that exactly $c_j$ members can be chosen from the block $(g_{k-j}, g_{k-j+1}]$ and then use \Cref{lemmasubgee} to get the following expression.

\begin{equation*}
V_b=V_{\{g_k,g_{k-1}\}} + \sum_{i=1}^{k-2}\left(  \sum_{\substack{ C+(1,1,0\dots,0)\in S_k\\ C_{+}=i}}\left(\prod_{j=3}^{k}\binom{a_j}{c_j} V_{g_k}V_{g_{k-1}}V_{g_{k-j+1}}V_{g_{k-j+1}-1}\dots \right) \right). 
\end{equation*}
Allowing $C_+=0$ enables us to rewrite the above equation  as
\begin{equation}
V_b=\sum_{i=0}^{k-2}\left(  \sum_{\substack{C+(1,1,0\dots,0)\in S_k\\ C_{+}=i}}\left(\prod_{j=3}^{k}\binom{a_j}{c_j} V_{g_k}V_{g_{k-1}}V_{g_{k-j+1}}V_{g_{k-j+1}-1}\dots \right) \right).     
\end{equation}
From 
\Cref{corollary:constant} we get, 
\begin{equation*}
    \begin{split}
        V_b & =\sum_{i=0}^{k-2}\left(  \sum_{\substack{C+(1,1,0\dots,0)\in S_k\\ C_{+}=i}}\left(\prod_{j=3}^{k}\binom{a_j}{c_j}\sum_{\substack{B+C+(1,1,0\dots,0)\in S_k\\ B_{+}=k-2-i}}~\prod_{q=1}^{k}\binom{a_q+b_q-2}{b_q} \right) \right)\\
        & = \sum_{C+(1,1,0\dots,0)\in S_k}\left(\prod_{j=3}^{k}\binom{a_j}{c_j}\sum_{\substack{B+C+(1,1,0\dots,0)\in S_k\\ B_{+}=k-2-C_{+}}}~\prod_{q=1}^{k}\binom{a_q+b_q-2}{b_q} \right).
    \end{split}
\end{equation*}

For $P=\{1\}$, the relation $(R3)$ of \Cref{hausmann-knutson} gives 
\begin{equation*}
R^m=\sum_{\substack{|S|\geq 1\\ 1\notin S}}V_S=\left(\sum_{\substack{ |S|\geq 1\\ \max(S)\leq g_k,~ 1\notin S}}V_S\right)+\sum_{j\in (g_k, b]}V_j.
\end{equation*}
If $j_1, j_2\in(g_k, b]$ such that $j_1\neq j_2$ then using \Cref{EqR3} and adding equations $E_{\{j_1\}}, E_{\{j_2\}}$ we get that $V_{\{j_1\}} + V_{\{j_2\}} =0$.
Hence all $V_j$'s are equal when $j\in(g_k, b]$. 
\begin{equation}\label{eqD}
R^m=\sum_{\substack{|S|\geq 1\\ 1\notin S}}V_S=\left(\sum_{\substack{ |S|\geq 1\\ \max(S)\leq g_k,~ 1\notin S}}V_S\right)+(b-g_k) V_b.
\end{equation}
Note that $C+(1,1,0,\dots,0)\in S_k$ implies $C=(0,0,c_3,\dots,c_k)$ and $\binom{a_1}{0}\binom{a_2}{0}=1$. Therefore, we can rewrite $V_b$ as 
\[V_b=\sum_{C+(1,1,0\dots,0)\in S_k}\left(\prod_{j=1}^{k}\binom{a_j}{c_j}\sum_{ \substack{B+C+(1,1,0\dots,0)\in S_k\\ B_{+}=k-2-C_{+}} }~\prod_{q=1}^{k}\binom{a_q+b_q-2}{b_q} \right).\]
Let $T=B+C$. Then it is easy to see that \[V_b=\sum_{\substack{T+(1,1,0\dots,0)\in S_k\\ |T|=k-2}}\left(\sum_{B\leq T}\prod_{j=1}^{k}\binom{a_j}{t_j-b_j}\binom{a_j+b_j-2}{b_j}\right).\]
Let $\Tilde{S}_{k-2}=\{0\}\times \{0\}\times S_{k-2}$. Now using binomial identity $\binom{a_j+b_j-2}{b_j}\equiv \binom{1-a_j}{b_j}$ we get, 
\[V_b=\sum_{\substack{T\in \Tilde{S}_{k-2}\\ |T|=k-2 }}\left(\prod_{j=1}^{k}\sum_{b_j}\binom{a_j}{t_j-b_j}\binom{1-a_j}{b_j}\right).\]
Now the Vandermonde identity $\binom{m+n}{r}=\sum_{k=0}^{r}\binom{m}{k}\binom{n}{r-k}$ gives 
\[V_b=\sum_{\substack{T\in \Tilde{S}_{k-2}\\ |T|=k-2}}\prod_{j=1}^{k}\binom{1}{t_j}.\]
Note that $(0,0,1,\dots,1)$ is the only possibile choice for $T$ in the above equation. 
Therefore, $V_b=1$.

Hence, using \Cref{corollary:constant} and \Cref{eqD} we get that \[R^m=\sum_{\substack{B\in S_k\\B_{+}=k}}\prod_{i=1}^{k}\binom{a_i+b_i-2}{b_i}+ (b-g_k).\] This completes the proof of \Cref{powerR}.
\end{proof}

The following is a straightforward corollary of the above result.

\begin{corollary}
Assuming the notations introduced in the proof of the previous result, and let $\langle \{g_1,\dots,g_k,n\}, \{b,n\}\rangle$ be the genetic code of $\alpha$. If 
$$\sum_{\substack{B\in S_k\\B_{+}=k}}\prod_{i=1}^{k}\binom{a_i+b_i-2}{b_i}+ (b-g_k) \equiv 0 \ \mathrm{(mod \  2)},$$ then $\malpha$ is tidy.
\end{corollary}
\begin{proof}
By \Cref{powerR}, we have that 
\[R^m=\sum_{\substack{B_{+}=k\\ B\in S_k}}\prod_{i=1}^{k}\binom{a_i+b_i-2}{b_i}+ (b-g_k) = 0.\]
Since the size of the smallest gee in $\langle \{g_1,\dots,g_k,n\}, \{b,n\}\rangle$ is $1$,  \Cref{lem:lower bound on coind} implies that \[n-4\leq \co(\malpha)\leq \hts(\malpha)\leq \ind(\malpha) \leq n-3.\] Using \cite[Proposition 2.2]{BUlowerdim}, 
the height of the Stiefel-Whitney class is full if and only if the index is full. Here $R^m =0$, thus $\ind(\malpha)$ is not full. In particular, $\ind(\malpha) \le n-4.$ This proves the result.
\end{proof}

For an illustration of \Cref{powerR}, we give an example where we compute the height of the Stiefel-Whitney class.
\begin{example}
\normalfont
Suppose that the genetic code is $\langle\{2,4,n\},\{b,n\}\rangle$ and $b\geq 5$. Note that the collection of subgees is 
\begin{equation}\label{swheightexample}
  \bigg\{\emptyset, \{1\}, \{2\}, \dots, \{b-1\}, \{b\}, \{1,2\}, \{1,3\}, \{1,4\}, \{2,3\}, \{2,4\}\bigg\}.  
\end{equation}

For the subgee $\{1,2\}$, the relation (R3) of \Cref{hausmann-knutson} gives 
\[R^m=\sum_{i=3}^{b}R^{m-1}V_i.\]
Using Proposition \ref{RViRVj}  and \Cref{swheightexample}, we get
\begin{equation}\label{swhtexample2}
R^{m}= (b-2)R^{m-1}V_b=b-2.   
\end{equation}
We now compute the value of $R^{m}$ using \Cref{powerR}. Note that only possible values of $B\in S_2$ are $(1,1)$ and $(2,0)$. 
Therefore, \Cref{powerR} gives \[R^{m}=\binom{1}{1}\binom{1}{1}+\binom{0}{0}\binom{2}{2}+b-4=b-2,\] as desired.
\end{example}

\section{The case of quasi-equilateral polygon spaces}\label{qepolys}

The planar polygon space $\malpha$ associated with the length vector $\alpha=(1,\dots,1,r)$ is called a \emph{quasi-equilateral} planar polygon space.
Suppose $r$ is a natural number. Then it is easy to see that $\alpha$ is generic if both $r$ and $n$ have the same parity. 
Moreover, if $r>n-1$ then $\malpha=\emptyset$. 
We make the following observations. 
\begin{enumerate}
    \item If $r=n-2$, then the genetic code of $\alpha$ is $\langle \{n\}\rangle$. It follows from \cite[Example 2.6]{geohausmann} $\malpha \cong S^{n-3}$ and $\mbalpha \cong \mathbb{R}P^{n-3}$. 
    \item If $r=n-4$, then the genetic code of $\alpha$ is $\langle \{n-1,n\}\rangle$. Then \cite[Example 2.12]{geohausmann} gives that $\malpha\cong \sharp_{n-1} (S^1\times S^{n-4})$ and $\mbalpha\cong \sharp_{n}\mathbb{R}P^{n-3}$. Here $\sharp_n X$ represents the connected sum of $n$ copies of a space $X$.
\end{enumerate}  

From \Cref{hausmann-knutson}, the cohomology ring  
 $H^{\ast}(\mbalpha; \Z_2)$ for $\alpha = (1,\dots, 1, r)$ is generated by classes $R, V_1,$ $ V_2, \dots, V_{n-1} \in H^1(\mbalpha; \Z_2)$ subject to the following relations:
\begin{enumerate}
\item[(R1)] $R V_i$ + $V_i^2=0$, for $i \in [n-1]$,
\item[(R2)] $V_S=0$ if $|S|\geq \frac{n-r}{2}$,
\item[(R3)] For $L\subseteq [n-1]$ such that $|L|\geq \frac{n+r}{2}$, \[\displaystyle\sum_{S\subseteq L} R^{|L-S|-1}V_S=0.\]
\end{enumerate}

Recall that the class $R$ coincides with the first Stiefel-Whitney class of the double cover $\malpha\to \mbalpha$.
Kamiyama \cite{KamiyamaSWheight} computed the height of $R$ in terms of the values of $n$ and $r$.
Before stating this result, we recall the following notations from \cite{KamiyamaSWheight}.

\textbf{Notations:}
\begin{enumerate}
\item $D(n):=n-2$, $e(n,r):=\frac{n-r}{2}-1$. Note that $e(n,r)$ is the largest size of the gee (here it is smallest as well since the genetic code has a single gene). 
\item $k(n,r) :=\max\bigg\{i : 0\leq i\leq e(n,r)-1, ~ \binom{D(n)-e(n,r)+i}{i}\equiv 1(\text{mod }2) \bigg\}$.
\item \[\phi(n,r) := \begin{cases}
n-3, & \text{ if } \binom{D(n)}{e(n,r)}\equiv 1(\text{mod }2) , \\[10pt]
\frac{n+r}{2}+k(n,r)-2, & \text{ if } \binom{D(n)}{e(n,r)}\equiv 0(\text{mod }2).
\end{cases}\]
\end{enumerate} 

\begin{theorem}[{\cite[Theorem A]{KamiyamaSWheight}}\label{KamiyamaSWHt}]
Let $h(n,r)$ be the height of the class $R$. 
Then, for all $n\geq 4$ and $r\in \mathbb{N}$ such that $r\leq n-2$ with the same parity as $n$, $h(n,r)=\phi(n,r)$.
\end{theorem}

We use \Cref{KamiyamaSWHt} to establish bounds on the index and the coindex of the corresponding $\malpha$'s by computing their $h(n,r)$. 
For that, we first give an algorithm to effectively and inductively compute the Stiefel-Whitney height of the corresponding quasi-equilateral planar polygon spaces.

Let $e(n,r) = k$ and that $2^t \le k < 2^{t+1}$ for some natural number $t$. For $1\le j \le k+1$,  denote the sets $\{n-2 :  \binom{n-2-(j-1)}{k-(j-1)}\equiv 1 \text{ (mod $2$)}\}$ (mod $2^{t+1}$) and $\{n-2 :  \phi(n,r)= n-3+(j-1)\}$ (mod $2^{t+1}$) by $A_j$ and $B_j$ respectively. For simplicity of notations, let $C_j=\sqcup_{i=1}^{j} B_i$ for any $1\le j\le k+1$. Our aim now is to inductively determine the set $B_j$ for $1\leq j\leq k+1$, starting with $B_1$. The following are some easy observations
\begin{itemize}
    \item[$(i)$] $A_1=B_1=C_1$.
    \item[$(ii)$] $B_j= A_j\setminus C_{j-1}$.
\end{itemize}

We first explain the method for computing $A_1$ and thereafter we  give arguments for computing $A_j$ inductively, which will be enough to compute $B_j$ and $C_j$ (from previous observations). Let $(k)_2 = k_t k_{t-1} \dots k_1 k_0$ be the binary representation of the number $k$, {\it i.e.}, $$k = 2^t k_t + 2^{t-1} k_{t-1} + \dots + 2k_1 + k_0.$$
Clearly, $k_0, \dots, k_t \in \{0,1\}.$ Let $I_k$ be the indexing set corresponding to $k$ that collects the location of non-zero $k_i$'s in an increasing order, {\it i.e.}, $I_k = \{i_1, \dots, i_s\} \subseteq [t] \cup \{0\}$ with $i_1 < i_2 < \dots < i_s = t$, so that $k_l = 1$ if and only if $l \in I_k$.

We notice that the height is full, {\it i.e.}, $\phi(n,r) = n-3$, if and only if $\binom{n-2}{k} \equiv 1 \ (\text{mod 2})$.
Using Lucas theorem, it suffices to know that in binary representations, each position in $(n-2)_2$ is at least as much as $(k)_2$, position-wise, {\it i.e.}, 
\begin{equation}\label{eq:lucas n-2}
    \alpha_l \ge k_l \text{ for all } 0 \le l \le t, 
\end{equation}
 where $\alpha_l$ is the coefficient of $2^l$ in the binary representation of  $n-2$.
Since $|I_k| = s$, there are $2^{t-s+1}$ choices for $n-2$ to satisfy \Cref{eq:lucas n-2} modulo $2^{t+1}$. These are exactly the places where $k_i$'s are zero. Let these values of $n-2$ be $\{n_1, \dots, n_{t-s+1}\} \text{ modulo }  2^{t+1}$, {\it i.e.}, $n-2 \equiv p$ for some $p \in \{n_1, \dots, n_{t-s+1}\} \text{ (mod  $2^{t+1}$)}$. Thus, $B_1=A_1= \{n_1, \dots, n_{t-s+1}\}$.

We now explain the algorithm of computing $A_{j+1}$, for $1 \le j\le k$, from $A_{j}$. Let $I_{k-(j-1)}=\{p_1,\dots,p_d\}$, where $p_d\leq t$. Recall that $I_{k-(j-1)}$ collects the location of $1$'s in the binary representation of $k-(j-1)$ in an increasing order. We do the computations depending upon whether $k-(j-1)$ is even or odd.

\noindent {\bf Case 1:} $k- (j-1)$ is odd.

We see that if $k-(j-1)$ is odd, then $p_1 = 0$. Therefore, $$I_{k-j} = I_{k-j+1} \setminus \{1\} = \{p_2, p_3, \dots, p_d\}.$$ 

Using Lucas theorem, we see that there are a total of $2^{t+1 - (d-1)}$ values that $n-2-j$ can take to satisfy the equation $ \binom{n-2-j}{k-j}\equiv 1 \text{ (mod $2$)}$. It is now easy to see that $\binom{n-2-j}{k-j}\equiv 1 \text{ (mod $2$)}$ if and only if $n-2-j \equiv q \text{ (mod } 2^{t+1})$ where $q\in A_j \cup \{\ell-1: \ell \in A_j\}$. Hence,  
\begin{equation}\label{eq:aj1comp odd}
\begin{split}
    A_{j+1}& = A_j \cup \{q+1 :  q \in A_j\} \text{ (mod } 2^{t+1}), \text{ and}\\
    B_{j+1} & = A_{j+1}\setminus C_{j}= \{q+1:q \in A_j\}\text{ (mod } 2^{t+1}).
\end{split}
\end{equation}

\noindent{\bf Case 2:} $k- (j-1)$ is even.

If $k- (j-1)$ is even, then $p_1>0$. 
Therefore 
$$I_{k-j} = \{0,1, \dots, p_1-1, p_2, \dots, p_d\}.$$

This gives us that the values of $n-2-j$ that satisfy $ \binom{n-2-j}{k-j}\equiv 1 \text{ (mod $2$)}$ are $n-2-j \equiv q$ (mod $2^{t+1}$), for some $q \in \{\ell_1 - 1: \ell_1 \in A_j, ~ \ell_1 \equiv 0 \text{ (mod } 2^{p_1})\}\cup \{\ell_2 \in A_j: \ell_2 \equiv 2^{p_1}-1 \text{ (mod } 2^{p_1})\}$. Therefore,
\begin{equation}\label{eq:aj1comp even}
\begin{split}
    A_{j+1} = & \big{\{}\{\ell_1: \ell_1 \in A_j,~ \ell_1-(j-1) \equiv 0 \text{ (mod } 2^{p_1})\}\\
    & \cup \{\ell_2+1: \ell_2 \in A_j,~ \ell_2-(j-1) \equiv 2^{p_1}-1 \text{ (mod } 2^{p_1})\}\big{\}} \text{ (mod } 2^{t+1}), \text{ and}\\
    B_{j+1} & =  \{\ell_2+1: \ell_2 \in A_j,~ \ell_2-(j-1) \equiv 2^{p_1}-1 \text{ (mod } 2^{p_1})\}\big{\}} \text{ (mod } 2^{t+1}).
\end{split}
\end{equation}

We repeat these steps to complete the analysis up till $\phi(n,r) = 1$. 

For a better understanding of the above algorithm, we provide the following example.
\begin{example}
\normalfont
    Let  $r = n-8$, {\it i.e.}, $e(n,r) = 3 = k$. 
Then $D(n) = n-2$ implying that $t = 1$ as $2^1 \le k=3 < 2^2$. 
Recall that
\begin{equation*}\label{phinn-8}
\phi(n,n-8) = \begin{cases}
n-3, & \text{ if } \binom{n-2}{3}\equiv 1(\text{mod }2) , \\[10pt]
n-6+k(n,n-8), & \text{ if } \binom{n-2}{3}\equiv 0(\text{mod }2).
\end{cases}    
\end{equation*}

From the binary representation $(3)_2  = 11$ of 3, we see that $3_0 = 3_1 = 1$ and $I_3 = \{i_1, i_2\} = \{0,1\}$ and hence $s=2$. 
Using the above algorithm, there are $2^{t-s+1} = 2^{1-2+1} = 1$ choices for $n-2 \text{ (mod $2^2$)}$ that satisfy $\binom{n-2}{3}\equiv 1(\text{mod }2)$. 
Therefore, $\phi(n,n-8) = n-3$, whenever $n-2 \equiv 3 \text{ (mod }2^2)$, {\it i.e.}, $n \equiv 1 \text{ (mod 4)}$. 
As per the notation in the above algorithm, $A_1 = B_1=C_1=\{3\}$.

We now compute when $\phi(n,n-8) = n-4$.
Since $k =3$ is odd, \Cref{eq:aj1comp odd} implies that $A_2 = \{0, 3\}$. So $B_2=\{0\}$, $n-2\equiv 0$ (mod 4). Therefore $\phi(n,n-8) = n-4$ if $n\equiv 2$ (mod 4). Note that $C_2=\{0,3\}$. 

To find the values of $n$ so that $\phi(n,n-8) = n-5$, we use \Cref{eq:aj1comp even} as $k-1 = 2$ is even. 
Here $p_1=1$ and $j=2$ implies that $A_3=\{1,3\}$. Hence, $B_3=\{1\}$. Therefore, $\phi(n,n-8) = n-5$ if $n-2 \equiv 1$ (mod 4) implying that $n \equiv 3$ (mod 4). Now $C_3=\{0,1,3\}$.

Finally, the values of $n$ that corresponds to $\phi(n,n-8) = n-6$ are given through \Cref{eq:aj1comp odd} as $k-2=1$ is odd.
Since $A_3 = \{1,3\}$, $A_4 = \{0,1,2,3\}$ implying that $B_4 = \{2\}$. Thus $n-2 \equiv  \text{ (mod 4)}$, {\it i.e.}, $\phi(n,n-8) = n-6$ whenever $n \equiv 0 \text{ (mod 4)}$. 
\end{example}


{In \Cref{table}, we note computations for the cases $r = n-6$ to $r = n-14$, {\it i.e.}, $e(n,r) = 2$ to $e(n,r) = 6$. 
The entry corresponding to the row labelled $n-i$ and the column labelled $j$ specifies the values that $n$, up to the appropriate modulo, can take for $\phi(n,r) = n-i$ and $e(n,r) = j$. 
For instance, for $e(n,r) = 5, \phi(n,r)$ is $n-5$ whenever $n$ is congruent to 3 module $2^3$. 
The empty cells of the table denotes that such cases do not arise, for example when $e(n,r) = 2$, the lowest value $\phi(n,r)$ can take is $n-5$. }

\begin{table}[H]
\setlength{\extrarowheight}{0.1cm}
\begin{tabular}{|c||c|c|c|c|c|}
\hline
\backslashbox{$\phi(n,r)$}{$e(n,r)$}
&\makebox[3em]{2}&\makebox[3em]{3}&\makebox[3em]{4}
&\makebox[3em]{5}&\makebox[3em]{6}\\\hline\hline

$n-3$ & 0,1   & 1 & 0,1,6,7 & 1,7 & 0,1 \\ \hline
$n-4$ & 2 & 2 & 2 & 0,2 & 2 \\ \hline
$n-5$ & 3 & 3 & 3 & 3 & 3 \\ \hline
$n-6$ & & 0 & 4 & 4 & 4 \\ \hline
$n-7$ & & & 5 & 5 & 5 \\ \hline
$n-8$ & & & & 6 & 6 \\ \hline
$n-9$ & & & & & 7 \\ \hline
\end{tabular}
\caption{Stiefel-Whitney height calculations}\label{table}
\end{table}

We now consider quasi-equilateral polygon spaces for $r=1 $ and $2$. We deal with these cases for specific values of $n$.
\begin{enumerate}
    \item Let $r=1$. Observe that for $\alpha$ to be generic, $n$ must be odd.
    \begin{enumerate}
        \item Consider $n=2^{s+1}-1$.
        
        In this case, the genetic code of $\alpha$ is $\langle\{2^s+1,\dots,2^{s+1}-1\}\rangle$ and the size of the gee is $2^s-2$.
Therefore, using \Cref{lem:lower bound on coind} and \cite[Proposition C]{KamiyamaSWheight} we get the following inequality
\[2^s-2\leq \co(\malpha)\leq \hts(\malpha)=2^s-2.\]
Therefore, $\co(\malpha)= \hts(\malpha)=2^s-2$.

\item Consider $n=2^s+1$.

From \cite[Proposition C]{KamiyamaSWheight}, we have $\hts(\malpha)=2^s-2$.
Here, the genetic code of $\alpha$ is $\langle\{2^{s-2}+2,\dots,2^s+1\}\rangle$ and the size of gee is $2^{s-1}-1$. Therefore, 
\[2^{s-1}-1\leq \co(\malpha)\leq \hts(\malpha)=2^s-2.\]
    \end{enumerate}
    
\item Let  $r=2$. Observe that for $\alpha$ to be generic, $n$ must be even.
\begin{enumerate}
    \item Consider $n=2^{s+1}-2$.
    
    From \cite[Proposition C]{KamiyamaSWheight}, we have $\hts(\malpha)=2^s-2$.
The genetic code of $\alpha$ in this case is $\langle\{2^s+1,\dots,2^{s+1}-2\}\rangle$ and the size of gee is $2^s-3$. Therefore, 
\[2^s-2\leq \co(\malpha)\leq \hts(\malpha)=2^s-2.\]
Therefore, $\co(\malpha)= \hts(\malpha)=2^s-2$.

\item Consider $n=2^s$.

From \cite[Proposition C]{KamiyamaSWheight}, we know that $\hts(\malpha)=2^s-3$. Therefore, $\ind(\malpha)=2^s-3$.
Here, the genetic code of $\alpha$ is $\langle\{2^{s-1}+2,\dots,2^s\}\rangle$. 
Since the size of gee is $2^{s-1}-2$, we have the following inequality
\[2^{s-1}-1\leq \co(\malpha)\leq \hts(\malpha)=2^s-3.\]
\end{enumerate}
\end{enumerate}

\section{Concluding remarks}

In most of the results of this paper, we have found the exact value of the index for various planar polygon spaces. 
However, in some cases, we could only find a lower bound for the index (for instance, {see \Cref{table}} ). It would be interesting to see if in those cases the height is equal to the index.

In \Cref{lem:lower bound on coind}, we showed that $\co(\malpha)\geq n-k-3$ for any generic $n$-length vector $\alpha$ with the smallest gee of size $k$. Based on our computations, we believe that this bound is tight. Therefore, we conjecture the following.

\begin{conjecture}\label{conj}
    Let $\alpha$ be a generic $n$-length vector. If the size of the smallest gee is $k$, then \[\co(\malpha)=n-k-3.\]
\end{conjecture}

In the proof of \Cref{gcan}, we showed that the \Cref{conj} is true for $\alpha=\{b,n\}$ if $n\neq 6$.

\section*{Acknowledgements} 
The authors would like to thank the anonymous referee for providing many insightful comments and suggestions that helped improve the quality and presentation of this article. In particular, the authors are thankful to the referee for suggestions to improve the proof of \cref{gcan},  pointing out a gap in the earlier proof of  \Cref{powerR} and for suggesting the use of Lucas theorem that simplified the calculations of \Cref{qepolys} significantly. 

\bibliographystyle{plain} 
\bibliography{main}

\begin{thebibliography}{10}

\bibitem{csorba}
P{\'e}ter Csorba.
\newblock {\em Non-tidy spaces and graph colorings}.
\newblock PhD thesis, ETH Zurich, 2005.

\bibitem{mfdav}
Donald~M. Davis.
\newblock Manifold properties of planar polygon spaces.
\newblock {\em Topology Appl.}, 250:27--36, 2018.

\bibitem{Davisformula}
Donald~M. {Davis}.
\newblock {On the cohomology classes of planar polygon spaces}.
\newblock In {\em Contemp math}, pages 85--89. American Mathematical Society,
  2018.

\bibitem{cohomologyclasses}
Donald~M. Davis.
\newblock On the cohomology classes of planar polygon spaces.
\newblock In {\em Topological complexity and related topics}, volume 702 of
  {\em Contemp. Math.}, pages 85--89. Amer. Math. Soc., [Providence], RI,
  [2018] \textcopyright 2018.

\bibitem{zbMATH05315240}
Michael Farber.
\newblock {\em Invitation to topological robotics}.
\newblock Zurich Lectures in Advanced Mathematics. European Mathematical
  Society (EMS), Z\"{u}rich, 2008.

\bibitem{gadgil}
Siddhartha Gadgil.
\newblock {Embedded spheres in {\{}{\{}$\backslash$(S{\^{}}2 $\backslash$times
  S{\^{}}1{\#} $\backslash$cdots {\#} S{\^{}}2 $\backslash$times
  S{\^{}}1$\backslash$){\}}{\}}}.
\newblock {\em Topology Appl.}, 153(7):1141--1151, 2006.

\bibitem{BUlowerdim}
Daciberg~L. Gon\c{c}alves, Claude Hayat, and Peter Zvengrowski.
\newblock The {B}orsuk-{U}lam theorem for manifolds, with applications to
  dimensions two and three.
\newblock In {\em Group actions and homogeneous spaces}, pages 9--28. Fak. Mat.
  Fyziky Inform. Univ. Komensk\'{e}ho, Bratislava, 2010.

\bibitem{cohomologyring}
J.-C. Hausmann and A.~Knutson.
\newblock The cohomology ring of polygon spaces.
\newblock {\em Ann. Inst. Fourier (Grenoble)}, 48(1):281--321, 1998.

\bibitem{geohausmann}
Jean-Claude Hausmann.
\newblock Geometric descriptions of polygon and chain spaces.
\newblock In {\em Topology and robotics}, volume 438 of {\em Contemp. Math.},
  pages 47--57. Amer. Math. Soc., Providence, RI, 2007.

\bibitem{HR1}
Jean-Claude Hausmann and Eugenio Rodriguez.
\newblock The space of clouds in {E}uclidean space.
\newblock {\em Experiment. Math.}, 13(1):31--47, 2004.

\bibitem{kamiyama10}
Yasuhiko Kamiyama.
\newblock {Homology of the universal covering of planar polygon spaces}.
\newblock {\em JP Journal of Geometry and Topology}, 10(2):171--181, 2010.

\bibitem{KamiyamaSWheight}
Yasuhiko Kamiyama and Kazufumi Kimoto.
\newblock The height of a class in the cohomology ring of polygon spaces.
\newblock {\em Int. J. Math. Math. Sci.}, pages Art. ID 305926, 7, 2013.

\bibitem{UBUthm}
Ji\v{r}\'{\i} Matou\v{s}ek.
\newblock {\em Using the {B}orsuk-{U}lam theorem}.
\newblock Universitext. Springer-Verlag, Berlin, 2003.
\newblock Lectures on topological methods in combinatorics and geometry,
  Written in cooperation with Anders Bj\"{o}rner and G\"{u}nter M. Ziegler.

\bibitem{mcc81}
Darryl McCullough.
\newblock {Connected sums of aspherical manifolds}.
\newblock {\em Indiana Univ. Math. J.}, 30:17--28, 1981.

\bibitem{BUmfds}
Oleg~R. Musin.
\newblock Borsuk-{U}lam type theorems for manifolds.
\newblock {\em Proc. Amer. Math. Soc.}, 140(7):2551--2560, 2012.

\bibitem{GP}
Gaiane Panina.
\newblock Moduli space of a planar polygonal linkage: a combinatorial
  description.
\newblock {\em Arnold Math. J.}, 3(3):351--364, 2017.

\bibitem{topgroup}
Peter Scott and Terry Wall.
\newblock Topological methods in group theory.
\newblock Homological group theory, {Proc}. {Symp}., {Durham} 1977, {Lond}.
  {Math}. {Soc}. {Lect}. {Note} {Ser}. 36, 137-203, 1979.

\end{thebibliography}

\end{document}